\documentclass[11pt]{article}

\usepackage[a4paper,top=1in,bottom=1in,left=1in,right=1in]{geometry}
\usepackage{amsmath,amssymb,amsthm}
\usepackage{mathtools}
\usepackage{enumitem}
\usepackage{microtype}
\usepackage{hyperref}

\newtheorem{theorem}{Theorem}[section]
\newtheorem{proposition}[theorem]{Proposition}
\newtheorem{lemma}[theorem]{Lemma}

\theoremstyle{definition}

\theoremstyle{remark}
\newtheorem{remark}[theorem]{Remark}

\newcommand{\Ann}{\operatorname{Ann}}
\newcommand{\Hom}{\operatorname{Hom}}
\newcommand{\lt}{\ltimes}

\title{Ideal Structures in Idealizations and Cardinalities of Annihilating Ideals}

\author{R. Nikandish\\
	Department of Mathematics\\
	Jundi-Shapur University of Technology\\
	Dezful, Iran\\
	\texttt{r.nikandish@ipm.ir}}
\date{}

\begin{document}
	
	\maketitle
	
	\begin{abstract}
	We describe the ideals of an idealization $D\ltimes M$ by triples $(I,N,\varphi)$, where $I$ is an ideal of $D$, $N$ is a submodule of $M$ with $IM\subseteq N$, and $\varphi\in\operatorname{Hom}_D(I,M/N)$. For a domain $D$ and torsion-free $M$, the nonzero proper ideals of $D\ltimes M$ with nonzero annihilator are exactly $0\ltimes N$. As an application, for $R=\mathbb R[[t]]\ltimes\mathbb R[[t]]$, the nonzero proper annihilating ideals form a countable set, while the set of all nonzero proper ideals has cardinality $2^{\aleph_0}$. This answers negatively a question of Behboodi and Rakeei.
	\end{abstract}
	
	
	\medskip
	
	
\section{Introduction}

Let $R$ be a commutative ring with identity. We denote by
\[
\mathcal I^*(R)
=
\{I\lhd R:0\neq I\neq R\}
\]
the set of nonzero proper ideals of $R$, and by
\[
\mathcal A^*(R)
=
\{I\in\mathcal I^*(R):\operatorname{Ann}_R(I)\neq0\}
\]
the set of nonzero proper ideals having nonzero annihilator.

The set $\mathcal A^*(R)$ is the vertex set of the
annihilating-ideal graph of $R$. This graph was introduced and
studied by Behboodi and Rakeei in \cite{BehboodiRakeei}. Two distinct
vertices $I$ and $J$ are adjacent whenever
\[
IJ=0.
\]
Thus, the cardinality of the vertex set of the annihilating-ideal
graph is precisely the cardinality of the family of nonzero proper
ideals having nonzero annihilator.

A natural question is whether, for a commutative ring which is not a
domain, the number of such ideals is always equal to the number of
all nonzero proper ideals. Behboodi and Rakeei proposed the following
conjecture in \cite{BehboodiRakeei}.

\medskip

\noindent
\textbf{Behboodi--Rakeei Conjecture.}
\emph{If $R$ is a commutative ring with identity which is not a
	domain, then}
\[
\boxed{
	|\mathcal A^*(R)|=|\mathcal I^*(R)|.
}
\]

\medskip

This is Conjecture~1.5 in the paper of Behboodi and Rakeei. The
purpose of the present paper is to show that this conjectured equality
does not hold in general.

Our approach is based on the idealization of a module. If $D$ is a
commutative ring and $M$ is a $D$-module, the idealization
\[
D\ltimes M
\]
is the abelian group $D\oplus M$ equipped with multiplication
\[
(d,m)(d',m')
=
(dd',dm'+d'm).
\]
Idealizations provide a natural class of commutative rings with
nontrivial zero-divisor structure and have been studied extensively;
see, for example, Huckaba \cite{Huckaba} and Anderson and Winders
\cite{AndersonWinders}.

The structure of ideals in an idealization is particularly useful for
our purpose. The standard description of homogeneous ideals is given
in terms of pairs $(I,N)$ satisfying
\[
IM\subseteq N.
\]
However, an arbitrary ideal of $D\ltimes M$ need not be homogeneous.
The additional information carried by an arbitrary ideal is encoded
by a homomorphism
\[
\varphi:I\longrightarrow M/N.
\]
In Section~2, we give an explicit description of arbitrary ideals of
$D\ltimes M$ in terms of triples
\[
(I,N,\varphi),
\qquad
\varphi\in\operatorname{Hom}_D(I,M/N).
\]
This description will provide the main tool for counting ideals.

We then consider the case in which $D$ is a domain and $M$ is a
torsion-free $D$-module. In this setting, we show that an ideal of
$D\ltimes M$ whose projection onto $D$ is nonzero has zero
annihilator. Consequently, the nonzero proper ideals having nonzero
annihilator are precisely those of the form
\[
0\ltimes N,
\]
where $N$ is a nonzero $D$-submodule of $M$.

This observation leads to a general cardinality criterion. If the
number of homomorphisms
\[
\operatorname{Hom}_D(I,M/N)
\]
for a suitable pair $(I,N)$ is larger than the number of nonzero
submodules of $M$, then the idealization $D\ltimes M$ provides a
counterexample to the Behboodi--Rakeei conjecture.

As a concrete application, we take
\[
D=\mathbb R[[t]]
\qquad\text{and}\qquad
M=D.
\]
The ring $D$ is a discrete valuation domain and has only countably
many ideals. Hence the nonzero proper annihilating ideals of
\[
R=\mathbb R[[t]]\ltimes\mathbb R[[t]]
\]
form a countably infinite family. On the other hand, by taking
\[
I=tD
\qquad\text{and}\qquad
N=tD,
\]
we obtain
\[
\operatorname{Hom}_D(tD,D/tD)
\cong D/tD
\cong\mathbb R.
\]
Thus there are $2^{\aleph_0}$ distinct ideals of $R$ arising from the
different homomorphisms
\[
\varphi\in\operatorname{Hom}_D(tD,D/tD).
\]
We prove that this is in fact the exact cardinality of the set of all
nonzero proper ideals of $R$. Consequently,
\[
|\mathcal A^*(R)|=\aleph_0
<
2^{\aleph_0}
=
|\mathcal I^*(R)|.
\]

Therefore, the Behboodi--Rakeei conjecture fails even for the
indecomposable idealization
\[
\mathbb R[[t]]\ltimes\mathbb R[[t]].
\]

	\section{Ideals in Idealizations}
	
	We begin with the ideal structure of an idealization.
	
	Let $D$ be a commutative ring and let $M$ be a $D$-module. Recall that
	the idealization $D\ltimes M$ has underlying additive group
	$D\oplus M$ and multiplication
	\[
	(d,m)(d',m')=(dd',dm'+d'm).
	\]
	The subset
	\[
	0\ltimes M=\{(0,m):m\in M\}
	\]
	is a square-zero ideal.
	
	The following theorem gives a description of arbitrary ideals of an
	idealization. The familiar homogeneous ideals correspond to the
	special case in which the homomorphism appearing below is zero.
	
	\begin{theorem}
		\label{thm:classification}
		Let $D$ be a commutative ring, let $M$ be a $D$-module, and put
		\[
		R=D\ltimes M.
		\]
		For an ideal $J$ of $R$, define
		\[
		I=\pi(J)
		\]
		where $\pi:R\to D$ is the natural projection, and put
		\[
		N=\{m\in M:(0,m)\in J\}.
		\]
		Then $I$ is an ideal of $D$, $N$ is a $D$-submodule of $M$, and
		\[
		IM\subseteq N.
		\]
		Moreover, there is a unique homomorphism
		\[
		\varphi\in\Hom_D(I,M/N)
		\]
		such that
		\[
		J
		=
		\{(d,m)\in D\ltimes M:
		d\in I,\ m+N=\varphi(d)\}.
		\]
		Conversely, every triple $(I,N,\varphi)$ satisfying
		\[
		I\lhd D,\qquad N\leq_D M,\qquad IM\subseteq N,
		\qquad
		\varphi\in\Hom_D(I,M/N)
		\]
		determines an ideal
		\[
		J_\varphi
		=
		\{(d,m)\in D\ltimes M:
		d\in I,\ m+N=\varphi(d)\}.
		\]
		Furthermore,
		\[
		\pi(J_\varphi)=I.
		\]
		Thus the ideals of $D\ltimes M$ are in bijection with the triples
		$(I,N,\varphi)$ satisfying the conditions above.
	\end{theorem}
	
	\begin{proof}
		Let $J$ be an ideal of $R$, and let
		\[
		I=\pi(J).
		\]
		Since $\pi$ is a ring homomorphism, $I$ is an ideal of $D$.
		
		Next define
		\[
		N=\{m\in M:(0,m)\in J\}.
		\]
		It is immediate that $N$ is an additive subgroup of $M$. If
		$a\in D$ and $m\in N$, then
		\[
		(a,0)(0,m)=(0,am)\in J,
		\]
		so $am\in N$. Hence $N$ is a $D$-submodule of $M$.
		
		We claim that
		\[
		IM\subseteq N.
		\]
		Let $d\in I$ and $x\in M$. Since $d\in\pi(J)$, there is some
		$m\in M$ such that $(d,m)\in J$. Therefore
		\[
		(0,x)(d,m)=(0,dx)\in J.
		\]
		Hence $dx\in N$, proving that $IM\subseteq N$.
		
		For each $d\in I$, choose $m\in M$ such that $(d,m)\in J$, and
		define
		\[
		\varphi(d)=m+N\in M/N.
		\]
		This is well-defined. Indeed, if both $(d,m)$ and $(d,m')$ belong to
		$J$, then
		\[
		(0,m-m')=(d,m)-(d,m')\in J,
		\]
		and consequently $m-m'\in N$.
		
		The map $\varphi$ is $D$-linear. For $a\in D$ and $d\in I$, if
		$(d,m)\in J$, then
		\[
		(a,0)(d,m)=(ad,am)\in J.
		\]
		Thus
		\[
		\varphi(ad)=am+N=a(m+N)=a\varphi(d).
		\]
		
		We now prove that
		\[
		J
		=
		\{(d,m)\in D\ltimes M:
		d\in I,\ m+N=\varphi(d)\}.
		\]
		If $(d,m)\in J$, then $d\in I$ and, by definition,
		\[
		m+N=\varphi(d).
		\]
		Conversely, suppose $d\in I$ and
		\[
		m+N=\varphi(d).
		\]
		By the definition of $\varphi$, there is some $m_0\in M$ such that
		\[
		(d,m_0)\in J
		\qquad\text{and}\qquad
		m_0+N=\varphi(d).
		\]
		Hence $m-m_0\in N$, so
		\[
		(0,m-m_0)\in J.
		\]
		Therefore
		\[
		(d,m)=(d,m_0)+(0,m-m_0)\in J.
		\]
		
		For the converse construction, suppose that $I$, $N$, and $\varphi$
		satisfy the stated conditions. Let
		\[
		J_\varphi
		=
		\{(d,m):d\in I,\ m+N=\varphi(d)\}.
		\]
		It is clearly an additive subgroup of $D\ltimes M$. To prove that it
		is an ideal, let
		\[
		(d,m)\in J_\varphi
		\qquad\text{and}\qquad
		(a,x)\in D\ltimes M.
		\]
		Then
		\[
		(d,m)(a,x)=(ad,am+dx).
		\]
		Since $I$ is an ideal, $ad\in I$. Moreover,
		\[
		dx\in IM\subseteq N.
		\]
		Consequently,
		\[
		am+dx+N=am+N.
		\]
		Since $\varphi$ is $D$-linear,
		\[
		\varphi(ad)=a\varphi(d)=a(m+N)=am+N.
		\]
		Here the first component in the product is $ad$, while the second
		component is $am+dx$ when the factors are written in the order
		$(a,x)(d,m)$. Thus, using commutativity of $D$, we obtain
		\[
		am+dx+N=\varphi(ad),
		\]
		and hence
		\[
		(a,x)(d,m)\in J_\varphi.
		\]
		Therefore $J_\varphi$ is an ideal.
		
		Finally, for every $d\in I$, the definition of $J_\varphi$ provides
		an element of $J_\varphi$ whose first component is $d$. Hence
		\[
		\pi(J_\varphi)=I.
		\]
		The uniqueness of $\varphi$ follows immediately from the equality
		\[
		\varphi(d)=m+N
		\]
		for any $(d,m)\in J$. This completes the proof.
	\end{proof}

	\begin{remark}
		\label{rem:literature}
		The idealization construction and the structure of its homogeneous
		ideals have been studied extensively. In particular, Anderson and
		Winders~\cite{AndersonWinders}, Theorem~3.1, characterize the ideals
		of the form
		\[
		I\ltimes N
		\]
		and show that such a set is an ideal precisely when
		\[
		IM\subseteq N.
		\]
		Their paper also discusses the distinction between homogeneous and
		arbitrary ideals of an idealization. See also Huckaba~\cite{Huckaba},
		Section~25, for the classical treatment of idealization.
		
		The parametrization in Theorem~\ref{thm:classification} retains the
		additional homomorphism
		\[
		\varphi\in\Hom_D(I,M/N),
		\]
		which is essential for the cardinality argument developed below.
	\end{remark}

	We next identify the ideals with nonzero annihilator in an important
	class of idealizations.
	
	\begin{proposition}
		\label{prop:ann}
		Let $D$ be a domain and let $M$ be a torsion-free $D$-module. If $J$
		is an ideal of
		\[
		R=D\ltimes M
		\]
		whose projection onto $D$ is nonzero, then
		\[
		\Ann_R(J)=0.
		\]
		Consequently,
		\[
		\mathcal A^*(R)
		=
		\{0\ltimes N:0\neq N\leq_D M\}.
		\]
	\end{proposition}
	
	\begin{proof}
		Suppose that the projection of $J$ onto $D$ is nonzero. Then there
		exists
		\[
		(d,m)\in J
		\]
		with $d\neq0$. Let
		\[
		(a,x)\in\Ann_R(J).
		\]
		Since $(a,x)$ annihilates $(d,m)$, we have
		\[
		(a,x)(d,m)=(ad,am+dx)=(0,0).
		\]
		The first equality gives
		\[
		ad=0.
		\]
		Since $D$ is a domain and $d\neq0$, it follows that $a=0$. The second
		equality then becomes
		\[
		dx=0.
		\]
		Because $M$ is torsion-free and $d\neq0$, we obtain $x=0$. Hence
		\[
		\Ann_R(J)=0.
		\]
		
		It remains to consider ideals whose projection onto $D$ is zero. Such
		an ideal has the form
		\[
		J=0\ltimes N
		\]
		for some nonzero $D$-submodule $N$ of $M$. For every $x\in M$ and
		$n\in N$,
		\[
		(0,x)(0,n)=(0,0).
		\]
		Therefore
		\[
		0\ltimes M\subseteq\Ann_R(0\ltimes N),
		\]
		and hence $0\ltimes N$ has nonzero annihilator. The assertion follows.
	\end{proof}

	The preceding proposition gives a useful general criterion for
	producing a discrepancy between the cardinalities of all nonzero
	proper ideals and annihilating ideals.
	
	\begin{proposition}
		\label{prop:criterion}
		Let $D$ be a domain and let $M$ be a torsion-free $D$-module. Put
		\[
		\kappa=
		\left|
		\{0\neq N\leq_D M\}
		\right|.
		\]
		Suppose that there exist an ideal $I$ of $D$ and a submodule $N$ of
		$M$ such that
		\[
		0\neq I\neq D,
		\qquad
		IM\subseteq N,
		\]
		and
		\[
		|\Hom_D(I,M/N)|>\kappa.
		\]
		Then
		\[
		|\mathcal A^*(D\ltimes M)|
		<
		|\mathcal I^*(D\ltimes M)|.
		\]
	\end{proposition}
	
	\begin{proof}
		By Proposition~\ref{prop:ann},
		\[
		\mathcal A^*(D\ltimes M)
		=
		\{0\ltimes N':0\neq N'\leq_D M\}.
		\]
		Thus
		\[
		|\mathcal A^*(D\ltimes M)|=\kappa.
		\]
		
		Now fix $I$ and $N$ as in the hypothesis. For every
		\[
		\varphi\in\Hom_D(I,M/N),
		\]
		Theorem~\ref{thm:classification} gives an ideal
		\[
		J_\varphi
		=
		\{(d,m):d\in I,\ m+N=\varphi(d)\}.
		\]
		Since
		\[
		\pi(J_\varphi)=I
		\]
		and $0\neq I\neq D$, each $J_\varphi$ is a nonzero proper ideal of
		$D\ltimes M$. Moreover, if
		\[
		\varphi\neq\psi,
		\]
		then
		\[
		J_\varphi\neq J_\psi,
		\]
		because the homomorphism associated with an ideal in
		Theorem~\ref{thm:classification} is unique. Hence
		\[
		|\mathcal I^*(D\ltimes M)|
		\geq
		|\Hom_D(I,M/N)|
		>
		\kappa.
		\]
		Therefore
		\[
		|\mathcal A^*(D\ltimes M)|
		<
		|\mathcal I^*(D\ltimes M)|.
		\]
	\end{proof}

	\section{A Cardinality Counterexample}
	
	We now apply the preceding criterion to a concrete idealization.
	
	Let
	\[
	D=\mathbb R[[t]]
	\]
	and consider
	\[
	R=D\ltimes D.
	\]
	Thus the elements of $R$ are pairs $(f,g)$ with
	$f,g\in\mathbb R[[t]]$, and multiplication is given by
	\[
	(f,g)(h,k)
	=
	(fh,fk+hg).
	\]
	
	The ring $D=\mathbb R[[t]]$ is a discrete valuation domain. Its ideals
	are precisely
	\[
	0,\quad tD,\quad t^2D,\quad t^3D,\quad\ldots,\quad D.
	\]
	In particular, $D$ has only countably many ideals.
	
	We first record the cardinality of its finite-length quotients.
	
	\begin{lemma}
		\label{lem:quotients}
		For every positive integer $s$,
		\[
		D/t^sD\cong\mathbb R^s
		\]
		as $\mathbb R$-vector spaces. Consequently,
		\[
		|D/t^sD|=2^{\aleph_0}.
		\]
		Moreover,
		\[
		|D|=2^{\aleph_0}.
		\]
	\end{lemma}
	
	\begin{proof}
		Every element of $D/t^sD$ has a unique representative of the form
		\[
		a_0+a_1t+\cdots+a_{s-1}t^{s-1},
		\qquad a_i\in\mathbb R.
		\]
		Hence
		\[
		D/t^sD\cong\mathbb R^s
		\]
		as $\mathbb R$-vector spaces. Since $\mathbb R$ has cardinality
		$2^{\aleph_0}$ and $s$ is finite,
		\[
		|D/t^sD|=|\mathbb R^s|=2^{\aleph_0}.
		\]
		
		Finally,
		\[
		D=\mathbb R[[t]]
		\]
		is in bijection with the set of all sequences of real numbers, so
		\[
		|D|
		=
		|\mathbb R^{\mathbb N}|
		=
		(2^{\aleph_0})^{\aleph_0}
		=
		2^{\aleph_0}.
		\]
	\end{proof}

	\subsection{The annihilating ideals}
	
	Since $D$ is a domain and is torsion-free as a module over itself,
	Proposition~\ref{prop:ann} applies immediately.
	
	\begin{proposition}
		\label{prop:ann-example}
		For
		\[
		R=\mathbb R[[t]]\lt\mathbb R[[t]],
		\]
		the nonzero proper ideals having nonzero annihilator are precisely
		\[
		0\ltimes t^nD,
		\qquad n\geq0.
		\]
		Consequently,
		\[
		|\mathcal A^*(R)|=\aleph_0.
		\]
	\end{proposition}
	
	\begin{proof}
		The nonzero ideals of $D$ are
		\[
		t^nD,
		\qquad n\geq0.
		\]
		By Proposition~\ref{prop:ann}, the nonzero proper ideals of $R$ with
		nonzero annihilator are exactly
		\[
		0\ltimes N
		\]
		with $N$ a nonzero ideal of $D$. Hence they are precisely
		\[
		0\ltimes t^nD,
		\qquad n\geq0.
		\]
		There are countably infinitely many such ideals.
	\end{proof}

	\subsection{A continuum of nonzero proper ideals}
	
	We now show that the full ideal lattice is much larger.
	
	Take
	\[
	I=tD
	\qquad\text{and}\qquad
	N=tD.
	\]
	Since $M=D$, we have
	\[
	IM=tD=N,
	\]
	so the conditions of Theorem~\ref{thm:classification} are satisfied.
	
	Moreover,
	\[
	tD\cong D
	\]
	as a $D$-module. Therefore
	\[
	\Hom_D(tD,D/tD)
	\cong
	\Hom_D(D,D/tD)
	\cong
	D/tD.
	\]
	Since
	\[
	D/tD\cong\mathbb R,
	\]
	we obtain
	\[
	|\Hom_D(tD,D/tD)|
	=
	2^{\aleph_0}.
	\]
	
	For every
	\[
	\varphi\in\Hom_D(tD,D/tD),
	\]
	define
	\[
	J_\varphi
	=
	\{(d,m)\in D\lt D:
	d\in tD,\ m+tD=\varphi(d)\}.
	\]
	By Theorem~\ref{thm:classification}, each $J_\varphi$ is an ideal of
	$R$. Furthermore,
	\[
	\pi(J_\varphi)=tD.
	\]
	Since
	\[
	0\neq tD\neq D,
	\]
	each $J_\varphi$ is a nonzero proper ideal of $R$.
	
	If
	\[
	\varphi\neq\psi,
	\]
	then
	\[
	J_\varphi\neq J_\psi
	\]
	by the uniqueness assertion in Theorem~\ref{thm:classification}.
	Consequently,
	\[
	|\mathcal I^*(R)|
	\geq
	2^{\aleph_0}.
	\]
	
	Thus the family of nonzero proper ideals is already uncountable, even
	though the family of nonzero proper annihilating ideals is countable.

	\subsection{The upper bound}
	
	For completeness, we establish that the total number of ideals of
	$R$ is at most continuum. This will give the exact cardinality of the
	ideal lattice.
	
	Let
	\[
	J
	\]
	be an ideal of $R=D\lt D$. By Theorem~\ref{thm:classification}, $J$
	is determined by a triple
	\[
	(I,N,\varphi)
	\]
	where $I$ and $N$ are ideals of $D$ satisfying
	\[
	I\subseteq N
	\]
	and
	\[
	\varphi\in\Hom_D(I,D/N).
	\]
	There are only countably many possibilities for $I$ and $N$.
	
	We examine the possible cardinalities of the homomorphism sets.
	
	If
	\[
	N=D,
	\]
	then
	\[
	D/N=0,
	\]
	so
	\[
	|\Hom_D(I,D/N)|=1.
	\]
	
	If
	\[
	N=0,
	\]
	the condition $I\subseteq N$ forces
	\[
	I=0,
	\]
	and again there is only one homomorphism.
	
	Finally, suppose
	\[
	N=t^sD
	\qquad(s\geq1).
	\]
	If $I\neq0$, then
	\[
	I=t^rD
	\]
	for some $r\geq0$, and
	\[
	I\cong D
	\]
	as a $D$-module. Hence
	\[
	\Hom_D(I,D/N)
	\cong
	D/N.
	\]
	By Lemma~\ref{lem:quotients},
	\[
	|D/N|=2^{\aleph_0}.
	\]
	Thus, for every admissible pair $(I,N)$,
	\[
	|\Hom_D(I,D/N)|\leq2^{\aleph_0}.
	\]
	
	Since there are only countably many admissible pairs $(I,N)$, the
	classification theorem gives
	\[
	|\operatorname{Id}(R)|
	\leq
	\aleph_0\cdot2^{\aleph_0}
	=
	2^{\aleph_0},
	\]
	where $\operatorname{Id}(R)$ denotes the set of all ideals of $R$.
	
	Combining this with the lower bound obtained above gives the main
	result.
	
	\begin{theorem}
		\label{thm:main}
		Let
		\[
		R=\mathbb R[[t]]\lt\mathbb R[[t]].
		\]
		Then
		\[
		|\mathcal A^*(R)|=\aleph_0
		\]
		and
		\[
		|\mathcal I^*(R)|=2^{\aleph_0}.
		\]
		In particular,
		\[
		|\mathcal A^*(R)|
		<
		|\mathcal I^*(R)|.
		\]
		Consequently, the Behboodi--Rakeei conjecture does not hold in
		general.
	\end{theorem}
	
	\begin{proof}
		The equality
		\[
		|\mathcal A^*(R)|=\aleph_0
		\]
		follows from Proposition~\ref{prop:ann-example}.
		
		The construction based on
		\[
		I=N=tD
		\]
		gives
		\[
		|\mathcal I^*(R)|\geq2^{\aleph_0}.
		\]
		On the other hand, the classification in
		Theorem~\ref{thm:classification} and the preceding counting argument
		give
		\[
		|\mathcal I^*(R)|\leq2^{\aleph_0}.
		\]
		Hence
		\[
		|\mathcal I^*(R)|=2^{\aleph_0}.
		\]
		The strict inequality follows immediately.
	\end{proof}

	\section{Further Remarks}
	
	The example above illustrates a general phenomenon in idealizations.
	When $D$ is a domain and $M$ is torsion-free, the annihilating ideals
	are controlled entirely by the submodules of $M$ occurring in ideals
	of the form
	\[
	0\ltimes N.
	\]
	By contrast, arbitrary ideals with nonzero projection onto $D$ can
	carry an additional parameter
	\[
	\varphi\in\Hom_D(I,M/N).
	\]
	If this Hom-set is sufficiently large, it can produce many more
	ideals than are available among the annihilating ideals.
	
	This mechanism is not specific to $\mathbb R[[t]]$. The role of the
	coefficient field in the present example is to make the quotient
	\[
	D/tD
	\]
	uncountable while keeping the ideal lattice of $D$ countable. More
	generally, a similar construction can be considered for a discrete
	valuation domain whose residue field has sufficiently large
	cardinality. The comparison between the cardinality of the residue
	field and the number of nonzero submodules of the chosen module then
	determines whether the same argument applies.
	
	It is also worth noting that the ring
	\[
	R=\mathbb R[[t]]\lt\mathbb R[[t]]
	\]
	is indecomposable. Indeed, a decomposition of a commutative ring as a
	nontrivial direct product is equivalent to the existence of a
	nontrivial idempotent. Suppose that
	\[
	(a,m)^2=(a,m).
	\]
	Then
	\[
	a^2=a
	\qquad\text{and}\qquad
	2am=m.
	\]
	Since $D=\mathbb R[[t]]$ is a local domain, its only idempotents are
	$0$ and $1$. If $a=0$, then $m=0$; if $a=1$, then
	\[
	2m=m,
	\]
	and hence again $m=0$. Thus the only idempotents of $R$ are
	\[
	(0,0)
	\qquad\text{and}\qquad
	(1,0).
	\]
	Therefore $R$ admits no nontrivial direct product decomposition.
	
	The example consequently shows that the failure of the
	Behboodi--Rakeei conjecture is not merely a consequence of taking a
	decomposable ring. The discrepancy is produced by the structure of
	arbitrary ideals in the idealization itself.

	\section{Concluding Remarks and Questions}
	
	The cardinality comparison considered in this paper suggests that
	the family of annihilating ideals and the full ideal lattice can
	behave quite differently, even for relatively well-controlled
	commutative rings. The idealization construction provides a
	particularly transparent setting in which this difference can be
	measured explicitly.
	
	The main point is that, for a domain $D$ and a torsion-free
	$D$-module $M$, the nonzero proper annihilating ideals of
	$D\ltimes M$ are exactly the ideals
	\[
	0\ltimes N.
	\]
	The remaining ideals may be parametrized by homomorphisms
	\[
	I\longrightarrow M/N.
	\]
	Thus the size of the Hom-sets
	\[
	\Hom_D(I,M/N)
	\]
	can be substantially larger than the number of nonzero submodules of
	$M$.
	
	For the particular ring
	\[
	\mathbb R[[t]]\lt\mathbb R[[t]],
	\]
	this produces the sharp cardinality comparison
	\[
	|\mathcal A^*(R)|=\aleph_0
	<
	2^{\aleph_0}
	=
	|\mathcal I^*(R)|.
	\]
	In particular, the Behboodi--Rakeei conjecture fails without any need
	to impose a decomposability assumption on the ring.
	
	Several related questions remain natural. One may ask for conditions
	on a domain $D$ and a torsion-free module $M$ under which
	\[
	|\mathcal A^*(D\lt M)|
	=
	|\mathcal I^*(D\lt M)|.
	\]
	Another question is to determine the possible pairs of cardinals
	\[
	\bigl(
	|\mathcal A^*(D\lt M)|,
	|\mathcal I^*(D\lt M)|
	\bigr)
	\]
	that can occur for idealizations. The parametrization in
	Theorem~\ref{thm:classification} suggests that these questions are
	closely connected with the cardinalities of the Hom-sets
	\[
	\Hom_D(I,M/N).
	\]
	A systematic study of these cardinal invariants may provide further
	examples in which the ideal lattice is substantially larger than the
	family of annihilating ideals.
	\section*{Use of generative AI}
	During the preparation of this manuscript, the author used Claude Sonnet 5
	(Anthropic) to check the correctness of the main counterexample, to identify
	errors in a proof and improve the clarity and language of the text. The author reviewed and edited all
	AI-assisted output and takes full responsibility for the content of the
	manuscript.

\end{document}